\documentclass[a4paper, 11 pt]{article}
\usepackage{amsmath, amssymb, amsthm}
\usepackage{mathtools}
\usepackage{graphicx}
\usepackage{aliascnt}
\usepackage{geometry}
\usepackage{fancyhdr}
\usepackage{euscript}
\usepackage{tikz,tkz-euclide,tikz-cd}
\usepackage{float}
\usepackage{caption}
\usepackage{subcaption}
\usepackage{csquotes}
\usepackage{iftex}
\ifPDFTeX
	\usepackage[utf8]{inputenc}
	\usepackage[T1]{fontenc}
	\usepackage[swedish, english]{babel}
	\usepackage{lmodern}
\else
	\ifXeTeX
		\usepackage{mathspec}
		\usepackage{polyglossia}
		\setmainlanguage{english}
	\fi
\fi
\usepackage[colorlinks=true, linkcolor=blue, citecolor=red]{hyperref}
\usetkzobj{all}

\newcommand{\loc}{\mathrm{loc}}

\newcommand*{\defeq}{\mathrel{\vcenter{\baselineskip0.65ex \lineskiplimit0pt
			\hbox{$\raisebox{-0.200ex}{\scriptsize.}$}\hbox{\scriptsize.}}}
	=} 
\newcommand{\co}{\colon\thinspace} 

\renewcommand{\d}{\partial}

\newcommand{\R}{\mathbb{R}}

\newcommand{\imp}{\ \Rightarrow \ }

\newcommand{\absv}[1]{\left\vert #1 \right\vert}
\newcommand{\norm}[1]{\left\Vert #1 \right\Vert}

\newtheorem{theorem}{Theorem}

\newaliascnt{lemma}{theorem}
\newtheorem{lemma}[lemma]{Lemma}
\aliascntresetthe{lemma}

\newaliascnt{proposition}{theorem}

\aliascntresetthe{proposition}

\newaliascnt{corollary}{theorem}
\newtheorem{corollary}[corollary]{Corollary}
\aliascntresetthe{corollary}

\theoremstyle{definition}

\newaliascnt{definition}{theorem}
\newtheorem{definition}[definition]{Definition}
\aliascntresetthe{definition}

\newaliascnt{example}{theorem}

\aliascntresetthe{example}

\newaliascnt{exercise}{theorem}

\aliascntresetthe{exercise}

\theoremstyle{remark}

\newaliascnt{remark}{theorem}

\aliascntresetthe{remark}

\numberwithin{equation}{section}

\title{On the invariant Cantor sets of period doubling type of infinitely renormalizable area-preserving maps}

\author{
Dan Strängberg
\thanks{Department of Mathematics, Uppsala University, 
Box 480, 75106 Uppsala (Sweden). {\tt dan.strangberg@math.uu.se}}
}

\makeindex

\begin{document}

\maketitle

\begin{abstract}
	Since its inception in the 1970's at the hands of Feigenbaum and, 
	independently, Coullet and Tresser the study of renormalization operators 
	in dynamics has been very successful at explaining universality 
	phenomena observed in certain families of dynamical systems. The first 
	proof of existence of a hyperbolic fixed point for renormalization of 
	area-preserving maps was given by Eckmann, Koch and Wittwer in 1984. 
	However there are still many things that are unknown in this setting, in 
	particular regarding the invariant Cantor sets of infinitely renormalizable 
	maps.
	
	In this paper we show that the invariant Cantor set of period doubling type 
	of any infinitely renormalizable area-preserving map in the universality 
	class of the Eckmann-Koch-Wittwer renormalization fixed point is always 
	contained in a Lipschitz curve but never contained in a smooth curve. 
	This extends previous results by de Carvalho, Lyubich and Martens about 
	strongly dissipative maps of the plane close to unimodal maps to the 
	area-preserving setting. The method used for constructing the Lipschitz 
	curve is very similar to the method used in the dissipative case but 
	proving the nonexistence of smooth curves requires new techniques.
\end{abstract}

\section{Introduction}

The study of renormalization techniques in dynamics began in the 1970's in 
independent efforts by Feigenbaum (\cite{Feigenbaum-1978}, 
\cite{Feigenbaum-1979}) and Coullet and Tresser (\cite{Tresser_Coullet-1978}) 
to explain the observed universality phenomena in families of maps on the 
interval undergoing period doubling bifurcation. Acting as a microscope, the 
renormalization operator can describe the geometric structure of the maps in 
question at smaller and smaller scales. The existence of a hyperbolic fixed 
point of the renormalization operator explains the observed universality. It 
has also been shown that the infinitely renormalizable maps, i.e. those 
contained in the stable manifold of the renormalization fixed point, have 
invariant Cantor sets. The dynamics of any two infinitely renormalizable maps 
restricted to their respective invariant Cantor sets is topologically 
conjugate. This naturally leads to the question of whether or not this 
conjugacy can be smooth. If there is always a smooth conjugacy we say that the 
invariant Cantor sets are \textit{rigid}. It turns out that for infinitely 
renormalizable unimodal maps of the interval the invariant Cantor sets are 
indeed rigid. 

Since the introduction of renormalization in dynamics this formalism has been 
generalized in different directions. It has been particularly successful in 
one-dimensional dynamics, explaining universality for unimodal maps, critical 
circle maps and maps with a Siegel disk (see e.g. \cite{Lyubich-1999}, 
\cite{Avila_Lyubich-2011}, \cite{Sullivan-1992}, 
\cite{deFaria_edson_deMelo-2006}, \cite{Martens-1998}, \cite{Yampolsky-2001}, 
\cite{Yampolsky-2002}, \cite{Gaidashev_Yampolsky-2016} and references therein).

The rigorous study of renormalization for dissipative two-dimensional systems 
was started by Collet, Eckmann and Koch in \cite{Collet_Eckmann_Koch-1981a}. There they define a renormalization 
operator for strongly dissipative Hénon-like maps and show that the 
one-dimensional renormalization fixed point is also a hyperbolic fixed point 
for nearby dissipative maps. This result explains observed universality in 
families of such maps. In a subsequent paper by Gambaudo, van Strien and Tresser 
\cite{Gambaudo_vanStrien_Tresser-1989} it is shown that the infinitely renormalizable maps, i.e. those contained in the stable manifold of the hyperbolic fixed point, have an invariant Cantor set on which the dynamics is 
conjugate to the dyadic adding machine. A different renormalization operator 
for strongly dissipative Hénon-like maps was defined by de Carvalho, Lyubich 
and Martens in \cite{Carvalho_Lyubich_Martens-2005}. For this operator they 
show, in addition to the previous results, that the invariant Cantor sets are 
not rigid. More precisely, they show that a topological invariant of 
infinitely renormalizable strongly dissipative Hénon-like maps called the 
\textit{average Jacobian} is an obstruction to rigidity. If two infinitely 
renormalizable maps have different average Jacobians the conjugacy between 
their respective invariant Cantor sets cannot be smooth. Instead there is a 
form of probabilistic rigidity. Moreover they show that there is no smooth 
curve containing the invariant Cantor set of any infinitely renormalizable map.
In \cite{Lyubich_Martens-2011_survey} Lyubich and Martens also state that every 
invariant Cantor set of this form is contained in a rectifiable curve. For a 
proof, see the preprint \cite{Lyubich_Martens-2011_preprint}.

In \cite{Eckmann_Koch_Wittwer-1984} Eckmann, Koch and Wittwer introduced a 
renormalization operator for area-preserving maps of period doubling type of 
the plane and proved, using computer assistance, the existence of a hyperbolic 
fixed point. This explains previously observed universality phenomena in 
families of such maps. Further investigations of this renormalization have been 
done by Gaidashev and Johnson in \cite{Gaidashev_Johnson-2009a}, 
\cite{Gaidashev_Johnson-2009b}, \cite{Gaidashev_Johnson-2016} and by Gaidashev, 
Johnson and Martens in \cite{Gaidashev_Johnson_Martens-2016}. In these papers 
they prove existence of period doubling invariant Cantor sets for all 
infinitely renormalizable maps and also show that they are rigid. Again this 
rigidity can be compared to the situation for dissipative maps where the 
average Jacobian is a topological invariant and an obstruction to rigidity. 
Since all area-preserving maps have the same average Jacobian, one would 
expect, if indeed it is such a classifying invariant, that these Cantor sets 
are rigid. In \cite{Gaidashev_Johnson_Martens-2016} a conjecture is made that 
the average Jacobian is indeed such a classifying invariant.

In the present paper we address two more issues concerning the invariant 
Cantor sets of infinitely renormalizable area-preserving maps that draw 
parallels to the situation for dissipative maps. More precisely, we will 
explore the existence of Lipschitz or smooth curves containing these sets. 
As in the dissipative case it turns out that there are always Lipschitz 
curves containing the invariant Cantor sets but there is never a smooth curve. 
The central parts of the proof of the nonexistence of smooth curves uses 
different methods from the dissipative case however since they are not 
applicable to the area-preserving case.

\subsection{Structure of the paper}

In Section 2 we begin by giving a short introduction to the renormalization of 
area-preserving maps where we define the necessary objects and state the known 
results that will be used in the rest of the paper. At the end of Section 2 we 
state the main results of this paper.

Section 3 deals with the existence of a Lipschitz curve containing the 
invariant Cantor set of each infinitely renormalizable area-preserving map.

Lastly, in Section 4 we prove that there are no smooth curves containing the 
invariant Cantor set of any infinitely renormalizable area-preserving map.

\section{Area-preserving renormalization}

We consider two slightly different renormalization schemes: the one 
introduced by Eckmann, Koch and Wittwer in \cite{Eckmann_Koch_Wittwer-1984} and 
the one used by Gaidashev, Johnson and Martens in 
\cite{Gaidashev_Johnson_Martens-2016}. Both schemes are defined for exact 
symplectic diffeomorphisms of subsets of $ \R^{2} $. The maps are also 
required to be reversible, i.e. $ T\circ F\circ T = F^{-1} $ where $ 
T(x,y)=(x,-y) $ for every $ (x,y)\in\R^{2} $, and satisfy the twist condition 
\[ \dfrac{\d X}{\d y} \neq 0 \] where $ (X,Y) = F(x,y) $. In 
\cite{Gaidashev_Johnson_Martens-2016} it is also required that $ F(0,0)=(0,0) 
$. For such maps it is possible to find a generating function $ S=S(x,X) $ such 
that
\[ \begin{pmatrix}
	x \\
	-S_{1}(x,X)
\end{pmatrix}\overset{F}{\mapsto}\begin{pmatrix}
	X \\
	S_{2}(x,X)
\end{pmatrix} \, , \]
where subscripts denote partial derivatives. It can be shown that $ S_{1}(X,x) 
= S_{2}(x,X)\equiv s(x,X) $ so that $ F $ can also be written as
\begin{equation}
	\begin{pmatrix}
		x \\
		-s(X,x)
	\end{pmatrix} \overset{F}{\mapsto}\begin{pmatrix}
		X \\
		s(x,X)
	\end{pmatrix}\, .
\end{equation}
Using this formulation it is possible to write the differential $ \mathrm{D}F $ 
in terms of $ s $ using implicit differentiation of $ y=-s(X,x) $ and $ 
Y=s(x,X) $. The result is

\begin{equation}
\label{eq:differential}
	\renewcommand*{\arraystretch}{2.3}
	\mathrm{D}F(x,y) = \begin{pmatrix}
	-\dfrac{s_{2}(X(x,y),x)}{s_{1}(X(x,y),x)} & 
	-\dfrac{1}{s_{1}(X(x,y),x)} \\
	s_{1}(x,X(x,y))-s_{2}(x,X(x,y))\dfrac{s_{2}(X(x,y),x)}{s_{1}(X(x,y),x)}
	& -\dfrac{s_{2}(x,X(x,y))}{s_{1}(X(x,y),x)}
	\end{pmatrix} \, .
\end{equation} 

Letting $ \Lambda[F](x,y)=(\lambda_{F} x,\mu_{F} y) $ and $ 
\psi_{0}[F](x,y)=(\lambda_{F} x+p_{F},\mu_{F} y) $ the two renormalization 
schemes are then defined by 
\begin{align}
	R_{EKW}(F) & = \Lambda^{-1}[F]\circ F^{2} \circ \Lambda[F] \\
	R_{GJM}(F) & = \psi_{0}^{-1}[F] \circ F^{2} \circ \psi_{0}[F]
\end{align}
respectively. Here $ \lambda_{F}, \mu_{F}, p_{F} $ depend analytically on the 
map $ F $. In the remainder of the paper we will suppress $ F $ to avoid 
notational clutter. 

The existence of a hyperbolic fixed point for the EKW 
renormalization operator was proven in \cite{Eckmann_Koch_Wittwer-1984} using 
computer assistance. Similarly the operator $ R_{GJM} $ also has a hyperbolic 
fixed point which is related to the fixed point of $ R_{EKW} $ by a 
translation, as we will see later in the paper. We denote the fixed points by $ 
F_{EKW} $ and $ F_{GJM} $ respectively. For the 
renormalization fixed points the values of the rescalings used to define the 
renormalization operators are 
\begin{align*}
	\lambda_{EKW} & =\lambda_{GJM}=\lambda=-0.249\dots\, , \\
	\mu_{EKW} & =\mu_{GJM}=\mu=0.061\dots\, ,
\end{align*}
see e.g. \cite{Gaidashev_Johnson-2016}. 
The Fréchet derivative $ \mathcal{D}R_{EKW} $ has two eigenvalues outside the 
unit circle. One of them corresponds to the universal scaling observed in 
families of area-preserving maps and the other is $ \lambda^{-1} $. The 
eigenvalue $ \lambda^{-1} $ turns out to be related to a translational change 
of variables and this eigenvalue is eliminated by using the rescalings of $ 
R_{GJM} $. Thus at the fixed point $ F_{EKW} $ the renormalization $ R_{EKW} $ 
has a codimension $ 2 $ stable manifold whereas at the fixed point $ F_{GJM} $ 
the renormalization $ R_{GJM} $ has a codimension $ 1 $ stable manifold. We 
denote the stable and unstable manifolds by $ W^{s}(F) $ and $ W^{u}(F) $ 
respectively. Maps $ F $ contained in either of their stable manifolds $ 
W^{s}(F_{EKW}) $  or $ W^{s}(F_{GJM}) $ will be called \textit{infinitely 
renormalizable} with respect to the corresponding renormalization operator.  

In addition to being defined for area-preserving maps of the plane, $ R_{EKW} $ 
is also defined on the generating functions themselves according to 
\begin{equation}
	R_{EKW}(s)(x,X) = \mu^{-1}s(z(x,X),\lambda X)
\end{equation}
where $ z $ is a symmetric function, $ z(x,X)=z(X,x) $, satisfying the equation
\begin{equation}
\label{eq:z-equation}
	s(\lambda x,z(x,X))+s(\lambda X,z(x,X)) = 0\, .
\end{equation}

The following normalizations are used for the EKW renormalization in 
\cite{Gaidashev_Johnson-2016} and will be of use to us:
\begin{itemize}
	\item $ s(1,0)=0, $
	\item $ s_{1}(1,0)=1, $
	\item $ z(1,0)=z(0,1)=1 \, . $
	\item $ \mu=z_{1}(1,0), $
\end{itemize}
For more details on this see, for example, \cite{Eckmann_Koch_Wittwer-1984} or 
\cite{Gaidashev_Johnson-2016}. 

For any $ F\in W^{s}(F_{GJM}) $ there are analytically defined simply connected 
domains $ B_{0}(F) $ and $ B_{1}(F) $ that are disjoint and satisfy 

\begin{itemize}
	\item $ F^{2}(B_{0}(F))\cap B_{0}(F)\neq \emptyset $, and
	\item $ \psi_{w}(B_{0}(RF)\cup B_{1}(RF))\subset B_{w}(F) $
\end{itemize} 
where $ w\in\{0,1\} $ and $ \psi_{1}=F\circ\psi_{0} $. It is then possible to 
define a nested sequence of sets $ B^{n}_{w} $ for $ w\in\{0,1\}^{n} $ by
\begin{align*}
	B^{n}_{w0} & = \psi^{n}_{w}(B_{0}(R^{n}F)) \\
	B^{n}_{w1} & = \psi^{n}_{w}(B_{1}(R^{n}F))
\end{align*}
where $ \psi^{n}_{w}=\psi_{w_{1}}[F]\circ\dots\circ\psi_{w_{n}}[R^{n-1}F] $ for $ 
w = \{w_{1},\dots,w_{n}\}\in\{0,1\}^{n} $.

The sets $ B^{n}_{w} $ are nested as follows:
\[ B^{n}_{wv}\subset B^{n-1}_{w} \]
for $ w\in\{0,1\}^{n} $ and $ v\in\{0,1\} $, see Lemma 3.3 of 
\cite{Gaidashev_Johnson_Martens-2016}. This gives us the following schematic 
picture of the renormalization microscope.

\begin{figure}[H]
	\centering
	\begin{tikzpicture}
		\node at (0,0) {\includegraphics[scale=0.6]{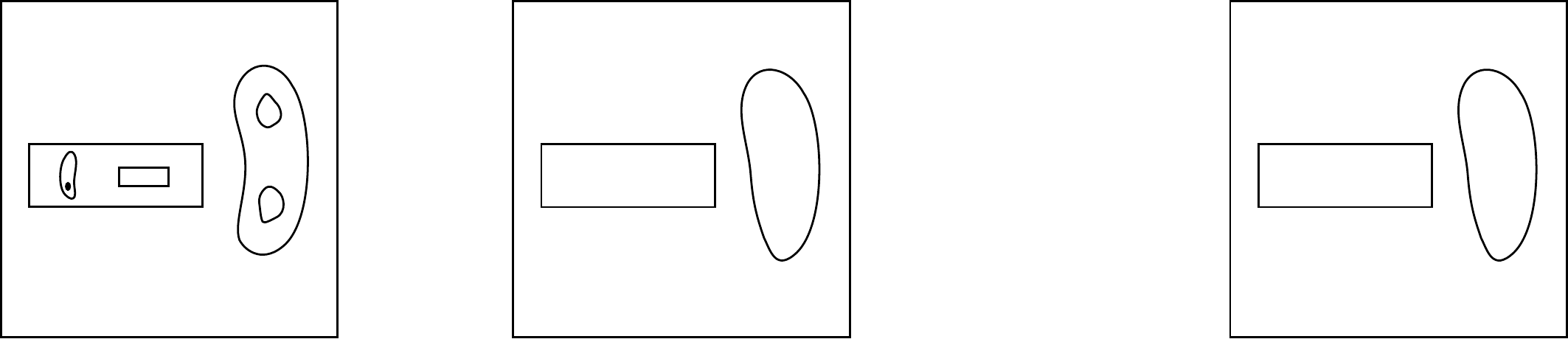}};
		\node at (-5,1.8) {$ F $};
		\node at (-0.8,1.8) {$ RF $};
		\node at (5,1.8) {$ R^{n}F $};
		\node at (-6,-0.7) {$ B_{0} $};
		\node at (-4.4,1.1) {$ B_{1} $};
		
		\draw[->,>=stealth] (-2,-0.9) to[out=200, in=300] 
			node[shift={(0.8,-0.3)}] {$ \psi_{0}[F] $} (-5.2,-0.4);
		\draw[->,>=stealth] (-2,0.9) to[out=160,in=40] 
			node[shift={(0,0.3)}] {$ \psi_{1}[F] $} (-4,0.7);
		\draw[->,>=stealth] (1.2,-0.9) to[out=200,in=300] (-1,-0.4);
		\draw[->,>=stealth] (1.2,0.9) to[out=160,in=40] (0.2,0.7);
		\draw[->,>=stealth] (4,-0.9) to[out=200,in=320] (2.8,-0.9);
		\draw[->,>=stealth] (4,0.9) to[out=160,in=20] (2.8,0.9);
		\draw[->,>=stealth] (4.0,-1.1) to[out=200,in=285] node[shift={(0,-0.3)}] 
			{$ \psi^{n}_{w}[F] $} (-5.9,-0.25);
		
		\node at (1.5,-0.9) {$ \dots $};
		\node at (1.5,0.9) {$ \dots $};
		\node at (2.5,-0.9) {$ \dots $};
		\node at (2.5,0.9) {$ \dots $};
	\end{tikzpicture}
	\caption{A schematic picture of the renormalization microscope.}
\end{figure}
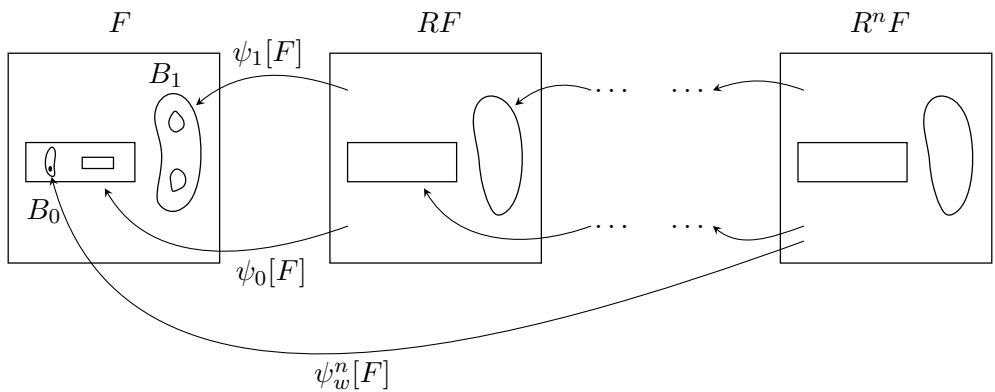

Using the nesting of the sets $ B_{w}^{n} $ we can make the following 
definition.
\begin{definition}
	The \textit{tip} of a map $ F\in W^{s}(F_{GJM}) $ is the point $ \tau = 
	\tau(F) = \bigcap_{n\ge 0}B^{n}_{0} $.
\end{definition}

In particular the tip $ \tau(F_{GJM}) $ of the renormalization fixed point will 
be of interest. It can be calculated as the fixed point of $ \psi_{0}(F_{GJM}) 
$.

\begin{align*}
	\tau(F_{GJM}) & = (x_{\tau},y_{\tau}) \\
	& = \psi_{0}(\tau(F_{GJM})) \\
	& = (\lambda x_{\tau} + p,\mu y_{\tau})
\end{align*} from which we can see that $ 
\tau(F_{GJM})=\left(\frac{p}{1-\lambda},0\right) $. The corresponding 
point for the EKW renormalization scheme is the origin $ (0,0) $ due to the 
absence of translation by $ p $ in $ \Lambda $. The nesting also allows us to 
define the set \[ \mathcal{O}_{F} =  
\bigcap_{n=1}^{\infty}\bigcup_{w\in\{0,1\}^{n}}B^{n-1}_{w} \] for any $ F\in 
W^{s}(F_{GJM}) $. In \cite{Gaidashev_Johnson_Martens-2016} it is proven that 
the set $ \mathcal{O}_{F} $ is an invariant Cantor set on which the dynamics of 
$ F $ is conjugate to the dyadic adding machine. These invariant Cantor sets 
are the subject of this paper.
\\

The following results will be needed in this paper.

\begin{lemma}[Lemma 7.1 of \cite{Gaidashev_Johnson-2009b}]
	\label{lem:prelim_contraction}
	For every $ F\in W(\rho) $, there exists a simply connected closed set $ 
	B_{F} $ such that $ B^{1}_{0}(F)\defeq \Lambda(B_{F})\subset B_{F} $ 
	and $ B_{1}^{1}(F)\defeq (F\circ\Lambda)(B_{F})\subset B_{F} $ are 
	disjoint, $ 
	F(B_{1}^{1}(F))\cap B_{0}^{1}(F)\neq \emptyset $ and 
	\[ \max \left\{\norm{\mathrm{D}\Lambda}_{B_{F}}, \, 
	\norm{\mathrm{D}(F\circ\Lambda)}_{B_{F}}\right\}\le 
	\theta\defeq 0.272\, . \]
\end{lemma}

Note that this lemma is formulated for the EKW renormalization scheme. Here, $ 
W(\rho) $ is the set of infinitely renormalizable area-preserving maps 
in a ball of radius $ \rho $ around an approximation of the fixed point for the 
EKW renormalization. A similar result is also valid for the renormalization 
scheme used in $ R_{GJM} $ with a slightly better 
bound, see Proposition 2.3 and Lemma 3.2 of 
\cite{Gaidashev_Johnson_Martens-2016}. However since the rescalings for the EKW 
renormalization scheme differ from the renormalization scheme used in 
\cite{Gaidashev_Johnson_Martens-2016} by a translation only depending 
on $ F $ the same bound on the differential is also valid in that setting.

Next lemma states that the rescalings $ \psi_{i}[F] $ converge uniformly to 
the rescalings of the fixed point $ F_{GJM} $ under iteration of the 
renormalization operator $ R_{GJM} $.

\begin{lemma}[Lemma 3.1 of \cite{Gaidashev_Johnson_Martens-2016}]
\label{lem:uniform_convergence}
	For every $ F\in W^{s}(F_{GJM}) $ and $ i\in\{0,1\} $ \[ 
	\lim_{n\to\infty}\norm{\psi_{i}[R^{n-1}_{GJM}F]-\psi_{i}[F_{GJM}]}_{C^{2}}=0.
	 \]
\end{lemma}

Lastly we will also need the rigidity result from 
\cite{Gaidashev_Johnson_Martens-2016}.

\begin{theorem}[Theorem 4.1 of \cite{Gaidashev_Johnson_Martens-2016}]
\label{thm:rigidity}
	The Cantor set $ \mathcal{O}_{F} $, with $ F\in W^{s}_{\loc}(F_{GJM}) $, is 
	rigid. Namely, There exists an $ \alpha_{0}>0 $ such that 
	\[ \mathcal{O}_{F} = \mathcal{O}_{F_{GJM}}\mod{(C^{1+\alpha})} \]
	for every $ 0<\alpha<\alpha_{0} $.
\end{theorem}

The constant $ \alpha_{0} $ appearing in \autoref{thm:rigidity} 
satisfies $ \alpha_{0}>0.237 $, see \cite{Gaidashev_Johnson_Martens-2016}.
\\

We will also need a property of twist maps called the \textit{ratchet 
phenomenon}. It says that for a twist map satisfying \[ \dfrac{\d X}{\d y} > 
a > 0 \] there are horizontal cones $ \Theta_{h} $ and vertical cones $ 
\Theta_{v} $ such that if $ p^{\prime}\in p+\Theta_{v} $ then $ 
F(p^{\prime})\in F(p)+\Theta_{h} $ and that the angle of the cones depend only 
on $ a $, see e.g. Lemma 12.1 of \cite{Gole-2001-book}. The same is true 
for negative twist maps with \[ \dfrac{\d X}{\d y} < a < 0 \, . \] This 
condition is clearly satisfied on $ B_{F} $ for all maps $ F $ considered here 
since it is compact. We will extend this to constant cone fields $ 
\Theta_{h}(p) $ and $ \Theta_{v}(p) $ so that at every point $ p $ $ 
\Theta_{h}(p) $ and $ \Theta_{v}(p) $ are just copies of $ \Theta_{h} $ and $ 
\Theta_{v} $ in the tangent space at $ p $. Using these cone fields we see that 
the ratchet phenomenon also implies that the differential of $ F $ maps the 
vertical cone field into the horizontal cone field in the corresponding tangent 
spaces. Thus 
\[ \mathrm{D}_{p} F(\Theta_{v}(p)) \subset \Theta_{h}(F(p)) \, . \] More 
precisely a positive twist map maps the half cone $ \Theta_{v}^{+}(p) $ into 
the half cone $ \Theta_{h}^{+}(F(p)) $ and $ \Theta_{v}^{-}(p) $ into $ 
\Theta_{h}^{-}(F(p)) $. A negative twist map changes the signs, i.e. $ 
\Theta_{v}^{+}(p) $ is mapped into $ \Theta_{h}^{-}(F(p)) $ and $ 
\Theta_{v}^{-}(p) $ is mapped into $ \Theta_{h}^{+}(F(p)) $.
\\

The results about the structure of the invariant Cantor sets of infinitely renormalizable area-preserving maps proven in this paper are the following two theorems.

\begin{theorem}
	\label{thm:lipschitzCurves}
	Every map $ F\in W^{s}(F_{GJM}) $ admits a Lipschitz curve containing its 
	invariant Cantor set $ \mathcal{O}_{F} $.
\end{theorem}

\begin{theorem}
	\label{thm:noSmoothCurves}
	There is no smooth curve containing $ \mathcal{O}_{F} $ for any $ F\in 
	W^{s}_{loc}(F_{GJM}) $.
\end{theorem}

These results extend those of \cite{Carvalho_Lyubich_Martens-2005} about 
nonexistence of smooth curves containing the invariant Cantor set 
for infinitely renormalizable dissipative maps to the area-preserving setting.

\section{Existence of Lipschitz curves}

In this section we prove \autoref{thm:lipschitzCurves}. The idea is to 
create a sequence of piecewise smooth curves with uniformly bounded Lipschitz 
constants that approach the invariant Cantor set. It then follows by the 
Arzelà-Ascoli theorem that there is a convergent subsequence. The limit of this 
subsequence is then our sought after Lipschitz curve. The sequence of curves is 
created inductively by choosing an initial curve $ \gamma_{0} $, projecting the 
previous curve using $ \psi_{0} $ and $ \psi_{1} $ and then connecting the 
pieces while at the same time controlling the Lipschitz constants.

Note that the proofs do not use the fact that the maps considered are exact 
symplectic twist maps. Rather the important part is that the renormalization 
microscope has a strong enough contraction to compensate for the 
reparametrization, allowing us to define the sequence of Lipschitz curves for 
the renormalization fixed point. The uniform convergence of the rescalings to 
the iterated function system of the renormalization fixed point then allows us 
to extend this result to all infinitely renormalizable maps. Thus a similar 
proof would also apply to other renormalization schemes or where a similar 
iterated function system appears, as long as we have appropriate bounds and 
convergence.

We begin by showing a bound on the Lipschitz constants on each of the projected 
pieces.

\begin{lemma}
\label{lem:contraction}
	Let $ F\in W^{s}_{loc}(F_{GJM}) $ and let $ \gamma\co [0,1]\to B_{F} $ be 
	a piecewise smooth curve with Lipschitz constant $ L $. Then for any $ n\ge 
	1 $ and any 
	$ w\in\{ 0,1 \}^{n} $ the curve $ \psi_{w}^{n}\circ \gamma $ has Lipschitz 
	constant at most $ \theta^{n}L $.
\end{lemma}

\begin{proof}
	Let $ s_{1},s_{2}\in [0,1] $. Then using \autoref{lem:prelim_contraction} 
	we get
	\begin{align*}
		\absv{\psi_{w}^{n}(\gamma(s_{1}))-\psi_{w}^{n}(\gamma(s_{2}))} & \le \int_{s_{2}}^{s_{1}}\absv{(\psi_{w}^{n}(\gamma(t)))^{\prime}}dt \\
		& = 
		\int_{s_{2}}^{s_{1}}\absv{\mathrm{D}\psi_{w}^{n}(\gamma(t))\gamma^{\prime}(t)}dt
		 \\
		& \le \theta^{n}\int_{s_{2}}^{s_{1}}\absv{\gamma^{\prime}(t)}dt \\
		& \le \theta^{n}L\absv{s_{1}-s_{2}} \, .
	\end{align*}
\end{proof}

We are now ready to prove the existence of the Lipschitz curve for the 
renormalization fixed point $ F_{GJM} $. As explained earlier we will achieve 
the result for all infinitely renormalizable maps as a corollary using the 
uniform convergence of $ \psi_{i}^{n}(F)\to \psi_{i}(F_{GJM}) $ from 
\autoref{lem:uniform_convergence}.

\begin{theorem}
	\label{thm:fixed_point_lipschitz}
	The Eckmann-Koch-Wittwer renormalization fixed point $ F_{GJM} $ admits a 
	Lipschitz curve containing its invariant Cantor set $ \mathcal{O}_{F_{GJM}} 
	$.
\end{theorem}

\begin{proof}
	Let $ \gamma_{0}\co [0,1]\to B_{F_{GJM}} $ be a piecewise smooth curve. We 
	will then construct a new curve $ \gamma_{1} $ in the following way: For 
	each $ w\in\{0,1\} $ consider the curve $ \psi_{w}\circ\gamma_{0} $. 
	Connect $ \gamma_{0}(0) $ to $ \psi_{0}(\gamma_{0}(0)) $, 
	$ \psi_{0}(\gamma_{0}(1)) $ to $ \psi_{1}(\gamma_{0}(0)) $ and $ 
	\psi_{1}(\gamma_{0}(1)) $ to $ \gamma_{0}(1) $ by line segments. This gives 
	us a piecewise smooth curve $ \gamma\co [0,5]\to B_{F_{GJM}} $. On each 
	piece of $ \gamma $ corresponding to a piece $ \psi_{w}\circ\gamma_{0} $ we 
	affinely reparametrize the curve on an interval of size $ \theta $, where $ 
	\theta=0.272 $ is the constant from \autoref{lem:prelim_contraction}. The 
	remaining parts of the interval $ [0,1] $ we use to 
	affinely reparametrize the $ 3 $ connecting pieces on intervals of size $ 
	\frac{1-2\theta}{3} $. This gives us a piecewise smooth curve $ \gamma_{1} 
	$ with Lipschitz constant $ L_{1} $ given by the maximum of $ 
	\gamma_{1}^{\prime}(t) $ over each smooth piece.
	
	Continuing by induction, assume we have constructed a piecewise smooth 
	curve $ \gamma_{k}\co [0,1]\to B_{F_{GJM}} $ and construct a curve $ 
	\gamma_{k+1} $ 
	using the same construction as for $ \gamma_{1} $. We would now like to 
	estimate the Lipschitz constant of $ \gamma_{k+1} $. First let $ 
	w\in\{0,1\} $ and consider the piece of the curve $ \gamma_{k+1} $ 
	given by $ \psi_{w}\circ\gamma_{k}\circ\varphi_{w} $ where $ \varphi_{w} $ 
	is the corresponding affine reparametrization. Using 
	\autoref{lem:contraction} we get
	\begin{align*}
		\absv{\gamma_{k+1}(s_{1})-\gamma_{k+1}(s_{2})} & = 
		\absv{\psi_{w}(\gamma_{k}(\varphi_{w}(s_{1})))-\psi_{w}(
		\gamma_{k}(\varphi_{w}(s_{2})))} \\
		& \le \theta L_{k}\absv{\varphi_{w}(s_{1})-\varphi_{w}(s_{2})} \\
		& = \dfrac{\theta L_{k}}{\theta}\absv{s_{1}-s_{2}} \\
		& = L_{k}\absv{s_{1}-s_{2}} \, .
	\end{align*} Hence on each part $ \psi_{w}\circ\gamma_{k}\circ\varphi_{w} $ 
	the curve $ \gamma_{k+1} $ has Lipschitz constant at most $ L_{k} $. 
	Outside of all $ B_{w}^{1}(F_{GJM}) $ the curve $ \gamma_{k+1} $ looks 
	exactly like $ \gamma_{1} $ and hence the Lipschitz constant here is at 
	most $ L_{1} $. From this we have 
	$ L_{k+1}\le \max(L_{k},L_{1})\le \max(L_{k-1},L_{1})\le\dots\le L_{1} $.
	
	We then have a sequence of Lipschitz curves $ \gamma_{k} $ with Lipschitz 
	constants uniformly bounded by $ L_{1} $. By the Arzelà-Ascoli theorem 
	there is then a subsequence $ \gamma_{k_{l}} $ converging uniformly to a 
	Lipschitz curve $ \gamma $. The image of this curve $ \gamma $ must then 
	contain the invariant Cantor set $ \mathcal{O}_{F_{GJM}} $ of $ F_{GJM} $ 
	since it intersects $ B^{n}_{w} $ for every $ w\in\{0,1\}^{n} $ and 
	every $ n $.
\end{proof}

\begin{figure}[H]
	\centering
	\begin{tikzpicture}
		\node at (0,0) {\includegraphics{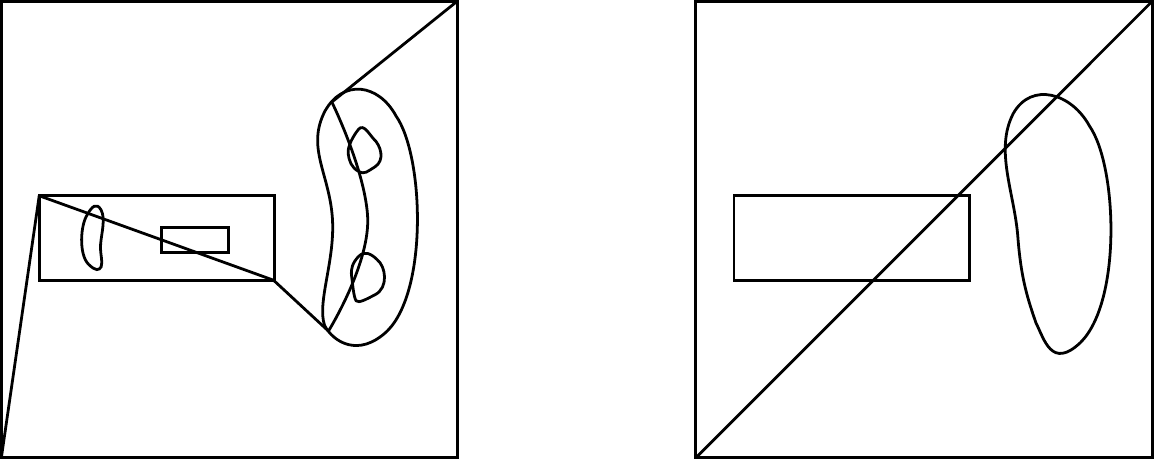}};
		\node at (-3.5,2.8) {$ B_{F_{GJM}} $};
		\node at (3.5,2.8) {$ B_{F_{GJM}} $};
		\node at (2.7,-1.2) {$ \gamma_{0} $};
		\node at (-3,-1) {$ \gamma_{1} $};
		
		\draw[->,>=stealth] (1.5,-1.5) to[out=200, in=300] 
			node[shift={(1,-0.3)}] {$ \psi_{0} $} (-4,-0.6);
		\draw[->,>=stealth] (1.5,1.5) to[out=160,in=40] 
			node[shift={(0,0.3)}] {$ \psi_{1} $} (-1.8,1.2);
	\end{tikzpicture}
	\caption{Constructing $ \gamma_{1} $ starting with $ \gamma_{0} $ diagonal.}
\end{figure}
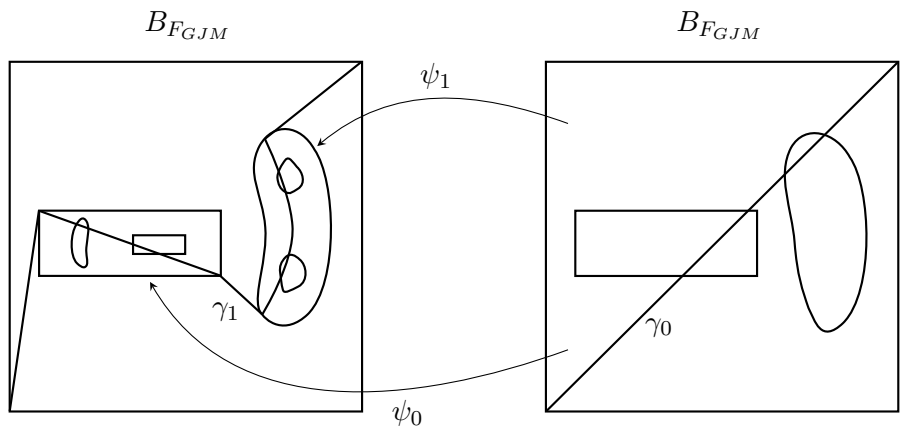

\begin{corollary}
	Every map $ F\in W^{s}(F_{GJM}) $ admits a Lipschitz curve containing its 
	invariant Cantor set $ \mathcal{O}_{F} $.
\end{corollary}

\begin{proof}
	By the uniform convergence of 
	$ \psi_{i}[R_{GJM}^{n-1}F]\to \psi_{i}[F_{GJM}] $ 
	given by \autoref{lem:uniform_convergence} we can find an $ N $, depending 
	only on $ F $, large enough so that for every $ n\ge N $ 
	\autoref{lem:prelim_contraction} is true for each $ 
	\psi_{i}\left[R_{GJM}^{n}F\right] $, possibly with a slightly larger 
	$ \theta $. Thus the construction from \autoref{thm:fixed_point_lipschitz} 
	over each $ B_{w}(R^{n}F) $ will not increase the Lipschitz constant for $ 
	n\ge N $. Similarly the curves connecting the pieces $ B_{w}(R^{n}F) $ will 
	also be $ C^{1} $-close to the connecting curves for the fixed point $ 
	F_{GJM} $ case so the Lipschitz constant will be uniformly bounded outside 
	$ B_{w}(R^{n}F) $ as well. Hence we again get a sequence of Lipschitz 
	curves $ \gamma_{k} $ with uniformly bounded Lipschitz constants $ L_{k} $. 
	By applying the Arzelà-Ascoli theorem we now get a Lipschitz curve 
	containing the invariant Cantor set $ \mathcal{O}_{F} $.
\end{proof}

Note that since the curve $ \gamma $ constructed above is Lipschitz it is in 
particular also rectifiable. In this sense we have proven the corresponding 
result of \cite{Carvalho_Lyubich_Martens-2005} for infinitely renormalizable 
area-preserving maps.

\section{Nonexistence of smooth curves}

We will now prove \autoref{thm:noSmoothCurves}. We will first consider only the 
renormalization fixed point $ F_{GJM} $. In order to show that there is no 
smooth curve containing $ \mathcal{O}_{F_{GJM}} $ we will first show that $ 
\mathcal{O}_{F_{GJM}} $ does not admit a continuous field of directions. 
Following the ideas of \cite{Carvalho_Lyubich_Martens-2005} this will then 
imply that $ \mathcal{O}_{F_{GJM}} $ does not admit a smooth curve containing 
it either. Using rigidity, the result for any $ F\in W^{s}(F_{GJM}) $ is then a 
simple corollary.

As opposed to the proof for Lipschitz curves this proof does use the fact that 
the fixed point is a twist map in an essential way.

We begin with a lemma about the EKW renormalization fixed point.

\begin{lemma}
	\label{lem:tipestimates}
	At the tip $ \tau = \bigcap_{n\ge 0}B^{n}_{0}(F_{EKW}) = (0,0) $ of the EKW
	renormalization fixed point $ F_{EKW} $ we have 
	\[ \dfrac{\partial X}{\partial x}(\tau)<0 \, . \]
\end{lemma}

\begin{proof}
	Note that the point $ (X,x)=(1,0) $ corresponds to the tip $ \tau $ of 
	$ F_{EKW} $ by the normalization condition $ s(1,0)=0 $.
	
	By the normalization condition $ s_{1}(1,0)=1 $ it follows that 
	\[ \dfrac{\partial X}{\partial y}(0,0) = -\dfrac{1}{s_{1}(1,0)}=-1 \] so 
	that the map $ F_{EKW} $ is in fact a negative twist map and we must have 
	that $ s_{1}(X,x)>0 $ everywhere. Since 
	\[ \dfrac{\partial X}{\partial x}(\tau) = 
	-\dfrac{s_{2}(1,0)}{s_{1}(1,0)} 
	\] we thus need to show that $ s_{2}(1,0) > 0 $.
	
	Consider first the equation for the renormalization fixed point for $ s $,
	\[ s(x,X) = \dfrac{1}{\mu}s(z(x,X),\lambda X) . \]
	Differentiating with respect to $ X $ gives us
	\[ s_{2}(x,X) = \dfrac{1}{\mu}(s_{1}(z(x,X),\lambda X)z_{2}(x,X)+\lambda 
	s_{2}(z(x,X),\lambda X)). \]
	Solving for $ s_{2} $, evaluating at the point $ (1,0) $ and using the 
	normalization condition $ z(1,0)=1 $ we get
	\[ s_{2}(1,0) = \dfrac{z_{2}(1,0)}{\mu-\lambda}\, , \]
	therefore showing that $ s_{2}(1,0) > 0 $ is reduced to showing that $ 
	z_{2}(1,0) > 0 $. 
	
	To do this we first differentiate Equation \ref{eq:z-equation} with respect 
	to $ y $ and get \[ s_{2}(\lambda x,z(x,X))z_{2}(x,X)+\lambda s_{1}(\lambda 
	X,z(x,X)) + s_{2}(\lambda X,z(x,X))z_{2}(x,X) = 0 \, . \]
	Evaluating at $ (x,X)=(1,0) $ yields 
	\begin{align*}
		 & s_{2}(\lambda,1)z_{2}(1,0) + \lambda s_{1}(0,1)+s_{2}(0,1)z_{2}(1,0) 
		 = 0 \\
		\imp & z_{2}(1,0) = -\lambda 
		\dfrac{s_{1}(0,1)}{s_{2}(\lambda,1)+s_{2}(0,1)} \, .
	\end{align*}
	Since $ -\lambda s_{1}(0,1)>0 $ due to the negative twist condition and $ 
	\lambda < 0 $ we must then show that $ s_{2}(\lambda,1)+s_{2}(0,1) $ is 
	also positive. To do this we differentiate Equation \ref{eq:z-equation} 
	with respect to $ x $ and get
	\[ \lambda s_{1}(\lambda x, z(x,X)) + s_{2}(\lambda 
	x,z(x,X))z_{1}(x,X)+s_{2}(\lambda X,z(x,X))z_{1}(x,X) = 0 . \]
	Evaluating at the point $ (x,X)=(1,0) $ then gives us
	\begin{align*}
		& \lambda s_{1}(\lambda,1)+\mu s_{2}(\lambda,1)+\mu 
		s_{2}(0,1) = 0 \\
		\imp & s_{2}(\lambda,1)+s_{2}(0,1) = 
		-\dfrac{\lambda}{\mu}s_{1}(\lambda,1) 
	\end{align*} which is positive by the negative twist condition forcing $ 
	s_{1} > 0 $ everywhere. Thus $ z_{2}(1,0) $ is positive and hence also $ 
	s_{2}(1,0) $ is positive, meaning that 
	\[ \dfrac{\partial X}{\partial x}(\tau) = 
	-\dfrac{s_{2}(1,0)}{s_{1}(1,0)} < 0 . \]
\end{proof}

We will now connect $ F_{EKW} $ with $ F_{GJM} $ by another lemma.

\begin{lemma}
	The maps $ F_{EKW} $ and $ F_{GJM} $ are conjugate by a translation in the 
	$ x $-direction.
\end{lemma}

\begin{proof}
	Recall that the renormalization operator $ R_{EKW} $ is defined by \[ 
	R_{EKW}(F) = \Lambda^{-1}\circ F^{2} \circ \Lambda \] where 
	$ \Lambda(x,y) = (\lambda x,\mu y) $. We begin by noticing that $ \psi_{0} 
	= 
	h^{-1}\circ \Lambda\circ h $ where $ h $ is the translation in the $ x 
	$-direction by the tip $ \tau=\frac{p_{F}}{1-\lambda} $, i.e. $ 
	h(x,y) 
	= \left(x-\frac{p}{1-\lambda},y\right) $. We then get:
	\begin{align*}
		h^{-1}\circ F_{EKW} \circ h & = h^{-1}\circ 
		R_{EKW}(F_{EKW})\circ h \\
		& = h^{-1}\circ \Lambda^{-1}\circ F_{EKW}^{2}\circ\Lambda\circ h\\
		& = \left(h^{-1}\circ\Lambda^{-1}\circ h\right)\circ\left(h^{-1}\circ 
		F^{2}_{EKW}\circ h\right)\circ\left(h^{-1}\circ\Lambda\circ h\right)\\
		& = \psi^{-1}_{0}\circ\left(h^{-1}\circ F^{2}_{EKW}\circ 
		h\right)\circ\psi_{0} \\
		& = R_{GJM}\left(h^{-1}\circ F_{EKW}\circ h\right)
	\end{align*}
	so $ F_{GJM}=h^{-1}\circ F_{EKW}\circ h $ is a fixed point of $ R_{GJM} $.
\end{proof}

Since such a conjugacy does not change the values of derivatives, $ F_{GJM} $ 
must also satisfy \[ \dfrac{\d X}{\d x} < 0 \] at the tip $ \tau = 
\left(\frac{p}{1-\lambda},0\right) $ of $ F_{GJM} $. Note also that 
this coordinate change corresponds to the generator $ \sigma^{1}_{-1,0} $ from 
\cite{Gaidashev_Johnson-2016}, which in turn corresponds to the eigenvalue $ 
\lambda^{-1} $ of $ \mathcal{D}R $ at the fixed point. With this we can 
now prove that the invariant Cantor set $ \mathcal{O}_{F_{GJM}} $ of the 
renormalization fixed point $ F_{GJM} $ does not admit a continuous invariant 
direction field.

\begin{theorem}
	\label{thm:noDirectionField}
	There is no continuous invariant direction field on the invariant Cantor 
	set $ \mathcal{O}_{F_{GJM}} $ of $ F_{GJM} $.
\end{theorem}

The basic idea of the proof is to assume to the contrary that there is such a 
continuous invariant direction field and show that it must then contain 
directions on either side of the vertical line. Using $ \psi_{0} $ we can 
project these towards the tip of $ F_{GJM} $ at which point these directions 
become horizontal in opposite directions, contradicting continuity of the field 
of directions.

\begin{proof}
	Assume, towards a contradiction, that there is a continuous invariant line 
	field 
	$ \theta $ on $ \mathcal{O}_{F_{GJM}} $. Let $ \Theta^{\pm}_{v}(p) $ and $ 
	\Theta^{\pm}_{h}(p) $ be the vertical and 
	horizontal cones at $ p\in\mathcal{O}_{F_{GJM}} $ from the ratchet 
	phenomenon. We then know, since $ F_{GJM} $ is a negative twist map, that $ 
	\mathrm{D}F_{GJM}(\Theta^{+}_{v}(p))\subset 
	\Theta^{-}_{h}(F_{GJM}(p)) $ and $ 
	\mathrm{D}F_{GJM}(\Theta^{-}_{v}(p))\subset 
	\Theta^{+}_{h}(F_{GJM}(p)) $. By reversibility we also have that $ 
	\mathrm{D}F^{-1}_{GJM}(\Theta^{+}_{v}(p))\subset 
	\Theta^{+}_{h}(F^{-1}_{GJM}(p)) $ is 
	the reflection in the horizontal axis of $ 
	\mathrm{D}F_{GJM}(\Theta^{-}_{v}(T(p))) $ 
	and in the same way $ \mathrm{D}F_{GJM}^{-1}(\Theta^{-}_{v}(p))\subset 
	\Theta^{-}_{h}(F_{GJM}^{-1}(p)) $ is the reflection in the horizontal axis 
	of $ \mathrm{D}F_{GJM}(\Theta^{+}_{v}(T(p))) $.
	
	Using this fact, suppose first that $ \theta(\tau) $ is not horizontal. 
	Since $ F^{2^{n}}_{GJM}\vert_{B_{0}^{n}} = \psi_{0}^{n}\circ 
	F_{GJM}\circ \psi_{0}^{-n}\vert_{B_{0}^{n}} $ we can find $ N $ large 
	enough so that for every $ n\ge N $ we have $ 
	\mathrm{D}_{\tau}(\psi_{0}^{-n})(\theta(\tau))\in\Theta^{\pm}_{v}(\tau) $. 
	Applying 
	$ \mathrm{D}F_{GJM} 
	$ and $ \mathrm{D}F_{GJM}^{-1} $ the resulting directions must then be on 
	opposite 
	sides of the vertical. Finally, projecting back to $ B_{0}^{n} $ with $ 
	\psi_{0}^{n} $ these directions approach the horizontal on opposite sides 
	of the vertical axis. Since $ \theta $ is invariant under $ F_{GJM} $ it 
	must 
	also be invariant under $ F^{2^{n}}_{GJM} $. Thus $ 
	\theta\left(F_{GJM}^{2^{n}}(\tau)\right) $ and $ 
	\theta\left(F_{GJM}^{-2^{n}}(\tau)\right) $ are close to the horizontal 
	axis 
	on opposite sides of the vertical axis. Since furthermore we can choose $ n 
	$ large enough so that $ F_{GJM}^{2^{n}}(\tau) $ and $ 
	F_{GJM}^{-2^{n}}(\tau) $ 
	are as close to $ \tau $ as we want this contradicts continuity of $ \theta 
	$.
	
	Next suppose that $ \theta(\tau) $ is horizontal. By 
	\autoref{lem:tipestimates} we have that $ \theta(F_{GJM}(\tau)) $ must then 
	be on the opposite side of the vertical axis. Using the same argument as in 
	the previous case we can then find points arbitrarily close to $ \tau $ 
	where $ \theta $ is on the other side of the vertical axis, again 
	contradicting continuity.
	
	We conclude that there can be no continuous invariant direction field.
\end{proof}

Following \cite{Carvalho_Lyubich_Martens-2005} we now prove that this also 
implies that there is no continuous invariant line field on 
$ \mathcal{O}_{F_{GJM}} $.

\begin{lemma}
	\label{lem:noDirectionField}
	The invariant Cantor set $ \mathcal{O}_{F_{GJM}} $ does not admit a 
	continuous invariant line field.
\end{lemma}

\begin{proof}
	Suppose towards a contradiction that there is a continuous invariant line 
	field on $ \mathcal{O}_{F_{GJM}} $. By compactness we can find an $ n_{0} $ 
	large enough so that the part of 
	the line field in $ \mathcal{O}_{F_{GJM}}\cap B^{n_{0}}_{w} $ can be 
	continuously oriented for every $ w\in\{0,1\}^{n_{0}} $. Furthermore, since 
	the pieces are all disjoint, choosing a continuous orientation on each of 
	them gives us a continuous orientation of the line field on the entire set 
	$ \mathcal{O}_{F_{GJM}}	$. This gives us a continuous direction field 
	on $ \mathcal{O}_{F_{GJM}} $ and it allows us to partition $ 
	\mathcal{O}_{F_{GJM}} $ into two parts: $ \mathcal{O}_{F_{GJM}}^{+} $ on 
	which $ F_{GJM} $ preserves the orientation and $ \mathcal{O}_{F_{GJM}}^{-} 
	$ on which $ F_{GJM} $ reverses the orientation. By 
	continuity both of these sets must be both closed and open. Then for every 
	large enough $ n $ every $ \mathcal{O}_{F_{GJM}}\cap B^{n}_{w} $ is 
	entirely contained in either $ 
	\mathcal{O}_{F_{GJM}}^{+} $ or $ \mathcal{O}_{F_{GJM}}^{-} $. Hence $ 
	F^{2^{n}}_{GJM}\vert_{\mathcal{O}_{F_{GJM}}\cap B^{n}_{0}} $ either 
	preserves or reverses orientation and therefore $ 
	F^{2^{n+1}}_{GJM}\vert_{\mathcal{O}_{F_{GJM}}\cap B^{n+1}_{0}} $ must 
	preserve orientation. Using the fact that $ F_{GJM} $ is the 
	renormalization fixed point this direction field can then be pulled back to 
	a new continuous invariant direction field on $ \mathcal{O}_{F_{GJM}} $, 
	but this contradicts Theorem \ref{thm:noDirectionField}. Therefore 
	$ \mathcal{O}_{F_{GJM}} $ cannot admit a continuous invariant line field.
\end{proof}

Finally, using the same idea as in \cite{Carvalho_Lyubich_Martens-2005}, we can 
now prove that there is no smooth curve containing $ \mathcal{O}_{F_{GJM}} $.

\begin{theorem}
	\label{thm:noSmoothCurve}
	There is no smooth curve containing $ \mathcal{O}_{F_{GJM}} $.
\end{theorem}

\begin{proof}
	Assume that there is a smooth curve $ \gamma $ containing $ 
	\mathcal{O}_{F_{GJM}} $. Then the line field $ \theta $ of tangent lines to 
	$ 
	\gamma $ gives us a continuous line field on $ \mathcal{O}_{F_{GJM}} $. 
	Furthermore it is invariant since it satisfies \[ \theta(p) = \lim_{q\to 
	p}l(p,q) \] where $ l(p,q) $ is the line passing through $ 
	p,q\in\mathcal{O}_{F_{GJM}} $. Using \autoref{lem:noDirectionField} this 
	contradicts \autoref{thm:noDirectionField}.
\end{proof}

Using \autoref{thm:rigidity} we get the following corollary of 
\autoref{thm:noSmoothCurve}.

\begin{corollary}
	There is no smooth curve containing $ \mathcal{O}_{F} $ for any $ F\in 
	W^{s}_{loc}(F_{GJM}) $.
\end{corollary}

\section*{Acknowledgements}

The author would like to thank Professor Marco Martens for his invaluable 
guidance and enlightening discussions during the authors stay at Stony Brook 
University as a visiting research scholar during the spring of 2016, during 
which time most of this work was done.

The authors visit to Stony Brook University was supported by STINT grant 
number 2012-2153.

\bibliographystyle{abbrv}
\bibliography{bibliography}

\end{document}